\documentclass[11pt,a4paper]{amsart}
\usepackage{amsfonts,amsmath,amssymb,amsthm,mathtools}
\usepackage{bbm}
\usepackage[textsize=tiny,shadow]{todonotes}
\numberwithin{equation}{section}

\DeclareMathOperator{\Ker}{Ker}
\DeclareMathOperator{\Ran}{Ran}
\DeclareMathOperator{\dist}{dist}

\DeclarePairedDelimiter{\abs}{|}{|}
\DeclarePairedDelimiter{\norm}{\lVert}{\rVert}

\newcommand{\eps}{\varepsilon}
\newcommand{\x}{z}
\newcommand{\sfuc}{\mathrm{sfuc}}

\newcommand{\dd}{\mathrm d}
\newcommand{\ee}{\mathrm e}
\newcommand{\ii}{\mathrm i}
\newcommand{\Id}{\mathrm{Id}}

\newcommand{\E}{\mathbb{E}}
\newcommand{\N}{\mathbb{N}}
\newcommand{\NN}{\mathbb{N}}
\renewcommand{\P}{\mathbb{P}}
\newcommand{\R}{\mathbb{R}}
\newcommand{\RR}{\mathbb{R}}
\newcommand{\ZZ}{\mathbb{Z}}
\newcommand{\one}{\mathbbm{1}}

\renewcommand{\Re}{\operatorname{Re}}

\newcommand{\cB}{{\mathcal B}}
\newcommand{\cF}{{\mathcal F}}
\newcommand{\cH}{{\mathcal H}}
\newcommand{\cU}{{\mathcal U}}

\newtheorem{theorem}{Theorem}[section]{\bf}{\it}
\newtheorem{hypothesis}[theorem]{Hypothesis}{\bf}{\it}
\newtheorem{lemma}[theorem]{Lemma}{\bf}{\it}
\newtheorem{corollary}[theorem]{Corollary}{\bf}{\it}
\newtheorem{proposition}[theorem]{Proposition}{\bf}{\it}
\theoremstyle{remark}
\newtheorem{remark}[theorem]{Remark}{\it}{\rm}

\renewcommand{\theenumi}{\alph{enumi}}

\title[Approximation for the control problem of the heat equation]{Exhaustion approximation for the control problem of the heat or
Schr\"odinger semigroup on unbounded domains}
\thanks{Filename  \jobname.tex }

\date{}

\author[A.~Seelmann]{Albrecht Seelmann}
\address{A.~Seelmann,
Technische Univer\-si\-t\"at Dortmund,
D-44227 Dortmund, Germany}
\email{albrecht.seelmann@math.tu-dortmund.de}

\author[I.~Veseli\'c]{Ivan Veseli\'c}
\address{I.~Veseli\'c,
Technische Univer\-si\-t\"at Dortmund,
D-44227 Dortmund, Germany}
\email{ivan.veselic@math.tu-dortmund.de}

\begin{document}

\begin{abstract}
We consider the control problem of the heat equation on bounded and unbounded domains, and more generally the corresponding
inhomogeneous equation for the Schr\"odinger semigroup. We show that if the sequence of null-controls associated to an exhaustion of
an unbounded domain converges, then the solutions do in the same way, and that the control cost estimate carries over to the
limiting problem on the unbounded domain. This allows to infer the controllability on unbounded domains by studying the control
problem on a sequence of bounded domains.
\end{abstract}

\maketitle

\section{Introduction}

Let $\cH$ and $\cU$ be Hilbert spaces, $H$ a lower semibounded self-adjoint operator on $\cH$, $B\colon\cU\to\cH$ a bounded linear
operator, and $T>0$.

We consider the abstract Cauchy problem
\begin{equation}\label{eq:abstrCauchy}
 \partial_t u(t) + H u(t) = Bf(t) \quad\text{ for }\quad 0<t<T,\quad u(0)=u_0,
\end{equation}
for given $u_0\in\cH$ and $f\in L^2((0,T),\cU)$, the~\emph{mild solution} of which is the continuous function $u\colon [0,T]\to\cH$
with
\begin{equation*}
 u(t) = \ee^{-tH}u_0 + \int_0^t \ee^{-(t-s)H} Bf(s)\,\dd s.
\end{equation*}
The mapping $\cB^T\colon L^2((0,T),\cU) \to \cH$ with
\begin{equation*}
 \cB^T f:= \int_0^T \ee^{-(T-s)H}Bf(s)\,\dd s
\end{equation*}
is called the~\emph{controllability map} for the system~\eqref{eq:abstrCauchy}. The system~\eqref{eq:abstrCauchy} is said to
be~\emph{null-controllable in time} $T>0$ if
\begin{equation*}
 \Ran \ee^{-TH} \subset \Ran \cB^T,
\end{equation*}
equivalently, if for every initial datum $u_0\in\cH$ there is a function $f\in L^2((0,T),\cU)$ with $u(T)=\ee^{-TH}u_0+\cB^T f=0$,
which explains the terminology. We call such a function $f$ a~\emph{null-control} for the initial datum $u_0$.

Note that the system~\eqref{eq:abstrCauchy} automatically is null-controllable if $B$ is surjective since then for given $u_0\in\cH$
one can choose a function $f\in L^2((0,T),\cU)$ with $Bf(t)=-\ee^{-tH}u_0/T$, and this function $f$ is a null-control.

It is well known that if the system~\eqref{eq:abstrCauchy} is null-controllable in time $T>0$, then for every initial datum $u_0$
there is a unique null-control of minimal norm in $L^2((0,T),\cU)$. Moreover, the mapping which assigns to each $u_0$ this unique
null-control is a bounded linear operator from $\cH$ to $L^2((0,T),\cU)$. We call this operator the~\emph{optimal feedback operator}
and denote it by $\cF^T$. We may now define the associated~\emph{control cost in time $T>0$} as
\begin{equation}\label{eq:controlCost}
 C_T := \norm{\cF^T} = \sup_{\norm{u_0}_\cH=1} \min\{ \norm{f}_{L^2((0,T),\cU)} \mid \ee^{-TH}u_0 + \cB^T f = 0 \}.
\end{equation}
In Appendix~\ref{sec:optimalFeedback} below we provide a more detailed background on the optimal feedback operator.

The question whether a given system is null-controllable and establishing estimates on the associated control cost are central
aspects of control theory, both in the context of abstract Cauchy problems as well as for partial differential equations. For a
broader discussion we refer, e.g., to the monographs \cite{FursikovI-96,Coron07,TW09} and the references therein.

We discuss the above situation in the setting where $H$ is an electromagnetic Schr\"odinger operator and $B$ is the multiplication
with a characteristic function. To this end, let $\Omega\subset\R^d$ be open, $\cH=\cU=L^2(\Omega)$, $\omega\subset\Omega$
measurable, and $B=\chi_\omega$. Moreover, let $A\in L_\text{loc}^2(\Omega,\R^d)$, and $V\colon\R^d\to\R$ such that
$V_+:=\max(V,0)\in L_\text{loc}^1(\Omega)$ and $V_-:=\max(-V,0)\colon \RR^d\to \RR$ is in the Kato class; see,
e.g.,~\cite[Section~4]{AS82} and also~\cite[Section~1.2]{CFKS87} for a discussion of the Kato class in $\R^d$.

Under these hypotheses, one can define the Dirichlet electromagnetic Schr\"odinger operator $H=H_\Omega(A,V)$ as a lower semibounded
self-adjoint operator on $L^2(\Omega)$ associated with the differential expression
\begin{equation*}
 (-\ii\nabla -A)^2 + V
\end{equation*}
via its quadratic form (with form core $C_c^\infty(\Omega)$). For details of this construction we refer to~\cite[Section~2]{HS04};
see also~\cite[Section~2]{BHL00} and the references therein. Note that the standard Dirichlet Laplacian $-\Delta_\Omega$ on $\Omega$
appears here as the particular case $-\Delta_\Omega=H_\Omega(0,0)$.

In this situation, the abstract Cauchy problem~\eqref{eq:abstrCauchy} reads
\begin{equation}\label{eq:abstrCauchyH}
 \partial_t u(t) + H_\Omega(A,V)u(t) = \chi_\omega\cdot f(t)\quad\text{ for }\quad 0<t<T,\quad u(0)=u_0.
\end{equation}
Observe that null-controls for this system can be assumed to be supported in $\omega$, that is, $f\in L^2((0,T),L^2(\omega))$. Note
also that this system automatically is null-controllable if $\omega=\Omega$ since then $B=\chi_\omega$ is surjective. Thus, in this
context, we may always assume that $\omega$ is a proper subset of $\Omega$.

We want to study the control problem for the system~\eqref{eq:abstrCauchyH} under varying domains $\Omega$, namely for an exhaustion
of a given domain in $\R^d$.

\subsection{The main result}\label{subsec:main}
Let $\Gamma\subset\R^d$ be open, $A\in L_\text{loc}^2(\Gamma,\R^d)$ and $V\colon\R^d\to\R$ such that
$V_+:=\max(V,0)\in L_\text{loc}^1(\Gamma)$ and $V_-:=\max(-V,0)$ is in the Kato class.

For $L>0$ we use the notation $\Lambda_L:=(-L/2,L/2)^d$ and $\Gamma_L:=\Gamma\cap\Lambda_L$. The sets $(\Gamma_L)_{L>0}$ clearly
give an exhaustion of $\Gamma$.

\begin{theorem}\label{thm:main}
 Let $S\subset\R^d$ be measurable, $\tilde u\in L^2(\Gamma)$, and $(L_n)_n$ a sequence in $(0,\infty)$ with $L_n\nearrow\infty$ as
 $n\to\infty$. Let $f_n\in L^2((0,T),L^2(\Gamma_{L_n}\cap S))$ for each $n\in\N$ be a null-control for the initial value
 problem~\eqref{eq:abstrCauchyH} on $\Omega=\Gamma_{L_n}$ with $\omega=\Gamma_{L_n}\cap S$ and initial value
 $u_0=\tilde u|_{\Gamma_{L_n}}$, and let $u_n$ be the corresponding mild solution.

 Suppose that $(f_n)_n$ converges weakly in $L^2((0,T),L^2(\Gamma))$ to $f\in L^2((0,T),L^2(\Gamma\cap S))$. Then, $f$ is a
 null-control for~\eqref{eq:abstrCauchyH} on $\Omega=\Gamma$ with $\omega=\Gamma\cap S$ and initial value $u_0=\tilde u$, and the
 corresponding mild solution is the weak limit of $(u_n)_n$ in $L^2((0,T),L^2(\Gamma))$.
\end{theorem}

Our technique to prove Theorem~\ref{thm:main} is not restricted to the specific choice of $(\Lambda_L)_{L>0}$. In fact, any
reasonable exhaustion of $\R^d$ would do instead, see Remark~\ref{rem:genExhaust} and Corollary~\ref{cor:semigroup} below. However,
the applications discussed in Section~\ref{sec:applications} below highly rely on this specific choice of the exhaustion.

\begin{remark}
 Parts of the convergence proof of Theorem~\ref{thm:main} provide quantitative error estimates, see Lemma~\ref{lem:semigroup} below.
 However, since we do not assume a speed of convergence of $(f_n)_n$ to $f$, we are also not able to give a complete error estimate
 of the corresponding convergence of $(u_n)_n$ to $u$. We conjecture that for null-controls $(f_n)_n$ with minimal norm the
 approximation error $f_n-f$ can be bounded efficiently. This will be explored in a later project.
\end{remark}

Note that if the null-controls $f_n$ in Theorem~\ref{thm:main} are uniformly bounded, that is,
\begin{equation}\label{eq:mainCrit}
 \norm{f_n}_{L^2((0,T),L^2(\Gamma_{L_n}\cap S))} \le c \quad\text{ for all }\quad n\in\N
\end{equation}
for some constant $c>0$, then $(f_n)_n$ has a weakly convergent subsequence with limit in $L^2((0,T),L^2(\Gamma\cap S))$.
Theorem~\ref{thm:main} can then be applied to every such weakly convergent subsequence, and the corresponding weak limit $f$ of the
subsequence of $(f_n)_n$ automatically satisfies the bound
\begin{equation}\label{eq:limitBound}
 \norm{f}_{L^2((0,T),L^2(\Gamma\cap S))}\le c.
\end{equation}

This leads to the following corollary to Theorem~\ref{thm:main}.

\begin{corollary}\label{cor:main}
 Let $S\subset\R^d$ be measurable, $\tilde u\in L^2(\Gamma)$, and $(L_n)_n$ a sequence in $(0,\infty)$ with $L_n\nearrow\infty$ as
 $n\to\infty$. Let $f_n\in L^2((0,T),L^2(\Gamma_{L_n}\cap S))$ for each $n\in\N$ be a null-control for the initial value
 problem~\eqref{eq:abstrCauchyH} on $\Omega=\Gamma_{L_n}$ with $\omega=\Gamma_{L_n}\cap S$ and initial value
 $u_0=\tilde u|_{\Gamma_{L_n}}$, and let $u_n$ be the corresponding mild solution.

 Assume that there is a constant $c\in\R$ such that~\eqref{eq:mainCrit} holds. Then there exists a subsequence of $(f_n)_n$ which
 converges weakly to a null-control $f\in L^2((0,T),L^2(\Gamma\cap S))$ for~\eqref{eq:abstrCauchyH} on $\Omega=\Gamma$ with
 $\omega=\Gamma\cap S$ and initial value $u_0=\tilde u$, satisfying~\eqref{eq:limitBound}. The mild solution $u$ associated to (any
 such weak accumulation point) $f$ is the weak limit of the corresponding subsequence of $(u_n)_n$ in $L^2((0,T),L^2(\Gamma))$.
\end{corollary}

Corollary~\ref{cor:main} in particular implies that estimates on the control cost~\eqref{eq:controlCost} with respect to
$\Gamma_{L_n}$ which are uniform in $n$ carry over to the limiting domain $\Gamma$.
It is Corollary~\ref{cor:main} that we will invoke in the applications below.

The rest of this note is organized as follows. 
In Section~\ref{sec:applications}, we apply Corollary~\ref{cor:main} to derive null-controllability for a (generalized) heat
equation on  $\Gamma=\R^d$ and appropriate choices of the control set $S$ from analogous results for the corresponding equation on
exhausting bounded domains.

Section~\ref{sec:semigroups} provides estimates on the difference of the action of Schr\"odinger semigroups on two different domains 
and an approximation result for a sequence of semigroups associated with an exhaustion of a given domain.
The following Section~\ref{sec:abstrCauchy} discusses the dependence of the solution of the abstract Cauchy problem on the
inhomogeneity. This is then combined with the mentioned semigroup approximation result to give an abstract convergence result which
contains Theorem~\ref{thm:main} as a special case.

Finally, Appendix~\ref{sec:optimalFeedback} provides some background on the optimal feedback operator.

\section{Applications}\label{sec:applications}
Here, we discuss in the case $\Gamma=\R^d$ particular choices for $A$, $V$, and $S$, where condition~\eqref{eq:mainCrit} can be
guaranteed. However, instead of $\R^d$, by a translation argument, our considerations apply just as well for every set $\Gamma$ that
can be exhausted with cubes such as the half-space $\Gamma=\R^{d-1}\times\R_+$ and the positive orthant $\Gamma=\R_+^d$.

\subsection{Control of the heat equation on thick sets}\label{sec:heat}

We consider the heat equation on $\R^d$,
\begin{equation}\label{eq:heatRd}
 \partial_t u(t) - \Delta_{\R^d} u(t) = \chi_S\cdot f(t) \quad\text{ for }\quad 0 < t < T,\quad u(0)=u_0,
\end{equation}
which is obtained from~\eqref{eq:abstrCauchyH} on $\Omega=\R^d$ with $A=0$, $V=0$, and $\omega=S\subset\R^d$. In this situation,
Theorem~3 in~\cite{EgidiV-18} states that~\eqref{eq:heatRd} is null-controllable if and only if $S$ is~\emph{thick} in $\R^d$ (see
also~\cite{WangWZZ}). The latter means that for some parameters $\gamma>0$ and $a=(a_1,\dots,a_d)\in\R_+^d$ one has
\begin{equation*}
 \abs{S\cap(x+[0,a_1]\times\dots[0,a_d])}\ge\gamma\prod_{j=1}^da_j
\end{equation*}
for all $x\in\R^d$, where $\abs{\,\cdot\,}$ denotes the Lebesgue measure on $\R^d$. In this case, $S$ is also said to be
$(\gamma,a)$-thick.

We have the following reformulation of Theorem~4 in~\cite{EgidiV-18}.

\begin{proposition}[{\cite[Theorem~4]{EgidiV-18}}]\label{prop:EV}
 Let $S\subset\R^d$ be a $(\gamma,a)$-thick set and $L\ge\max_{j=1,\dots,d}a_j$. Then, the system
 \begin{equation}\label{eq:heat}
  \partial_t u(t) - \Delta_{\Lambda_L} u(t) = \chi_{\Lambda_L\cap S}\cdot f(t) \quad\text{ for }\quad 0 < t < T,\quad u(0)=u_0,
 \end{equation}
 is null-controllable. Moreover, there is a universal constant $K>0$ such that for every initial datum $u_0\in L^2(\Lambda_L)$ there
 is a corresponding null-control $f\in L^2((0,T),L^2(\Lambda_L\cap S))$ satisfying
 \begin{equation*}
  \norm{f}_{L^2((0,T),L^2(\Lambda_L\cap S))}
  \le
  \tilde{K}^{1/2}\exp\Bigl(\frac{\tilde{K}}{2T}\Bigr)\cdot\norm{u_0}_{L^2(\Lambda_L)},
  \quad \tilde{K}=\Bigl(\frac{K^d}{\gamma}\Bigr)^{K(d+\norm{a}_1)},
 \end{equation*}
 where $\norm{a}_1=a_1+\dots+a_d$.
\end{proposition}
Recall that the Laplacian on $\Lambda_L$ in \eqref{eq:heat} is provided with Dirichlet boundary conditions.
Since for $u_0\in L^2(\R^d)$ one has $\norm{u_0}_{L^2(\Lambda_L)} \le \norm{u_0}_{L^2(\R^d)}$ for all $L>0$,
Proposition~\ref{prop:EV} in combination with Corollary~\ref{cor:main} allows us to reproduce the sufficiency part of Theorem~3
in~\cite{EgidiV-18} as part of the following corollary.

\begin{corollary}\label{cor:EV}
 Let $S\subset\R^d$ be thick, $\tilde u\in L^2(\R^d)$, and $(L_n)_n$ a sequence in $(0,\infty)$ with $L_n\nearrow\infty$ as
 $n\to\infty$. Then, there is a subsequence $(L_{n_k})_k$ of $(L_n)_n$ and null-controls $f_{n_k}$ for the initial value
 problem~\eqref{eq:heat} with $u_0=\tilde u|_{\Lambda_{L_{n_k}}}$ with corresponding mild solutions $u_{n_k}$ such that:
 \begin{enumerate}
  \renewcommand{\theenumi}{\roman{enumi}}
  \item $(f_{n_k})_k$ converges weakly in $L^2((0,T),L^2(\R^d))$ to a null-control for the initial value problem~\eqref{eq:heatRd}
        with $u_0=\tilde{u}$;
  \item the corresponding mild solution to the initial value problem~\eqref{eq:heatRd} is the weak limit of $(u_{n_k})_k$ in
        $L^2((0,T),L^2(\R^d))$.
 \end{enumerate}
 In particular, the system~\eqref{eq:heatRd} is null-controllable, and the associated control cost $C_T$ satisfies
 \begin{equation*}
  C_T \le \tilde{K}^{1/2}\exp\Bigl(\frac{\tilde{K}}{2T}\Bigr)
 \end{equation*}
 with $\tilde{K}$ as in Proposition~\ref{prop:EV}.
\end{corollary}

\subsection{Control of the generalized heat equation on equidistributed sets}

Similarly as above, we can use the main result of~\cite{NakicTTV-18} to infer the existence of uniformly bounded null-controls along
a sequence of increasing cubes. The result of~\cite{NakicTTV-18} allows more general semigroup evolutions than discussed in
Subsection~\ref{sec:heat} above, but requires more restrictions on the control set $S$. Let us formulate this precisely:

In the following we consider Schr\"odinger operators $H_\Omega(A,V)=H_\Omega(0,V)$ with vanishing magnetic $A\equiv 0$ and bounded
electric potential $V\in L^\infty(\Omega)$. In accordance with the approximation problem we have in mind the domain $\Omega$ will be
either $\RR^d$ or a (large) cube $\Lambda_L$.

Given $G,\delta > 0$, we say that a sequence $z_j \in \RR^d$, $j \in (G \ZZ)^d$, is~\emph{$(G,\delta)$-equidistributed}, if
\[
 \forall j \in (G \ZZ)^d \colon \quad  B(\x_j , \delta) \subset \Lambda_G + j .
\]
Corresponding to a $(G,\delta)$-equidistributed sequence $z_j$ we define for $L \in G \NN$
the~\emph{$(G,\delta)$-equidistributed}
sets
\[
 S_\delta = \bigcup_{j \in (G \ZZ)^d } B(\x_j , \delta) , \quad
 S_\delta (L) = S_\delta \cap \Lambda_L .
\]
The main novelty of~\cite{NakicTTV-18} is a spectral inequality or uncertainty principle.
\begin{proposition}[{\cite[Corollary~2.4]{NakicTTV-18}}] \label{thm:NTTVresult}
 There is $N = N(d)$ such that for all
 $G> 0$, all $\delta \in (0,G/2)$,
 all $(G,\delta)$-equidistributed sequences,
 all measurable and bounded $V\colon \RR^d \to \RR$,
 all $L \in G\NN$,
 all $E \geq 0$
 and all $\phi \in \mathrm{Ran} (\chi_{(-\infty,E]}(H_{\Lambda_L}))$ we have
 \begin{equation}
 \label{eq:sfuc}
 \lVert \phi \rVert_{L^2 (S_\delta (L))}^2
 \geq C_{\sfuc, G} \lVert \phi \rVert_{L^2 (\Lambda_L)}^2
 \end{equation}
 where
 \begin{equation*}
 C_{\sfuc, G} = C_{\sfuc, G} (d, \delta ,  E, \lVert V \rVert_\infty )
 :=  \left(\frac{\delta}{G} \right)^{N \bigl(1 + G^{4/3} \lVert V \rVert_\infty^{2/3} + G \sqrt{E} \bigr)} .
 \end{equation*}
\end{proposition}

>From this one obtains as a corollary the following result which is given in Section~2 of~\cite{NakicTTV-18}.
\begin{proposition}[{\cite[Theorem~2.15]{NakicTTV-18}}]
 \label{thm:contcost}
 For every $G > 0$, $\delta \in (0, G/2)$ and $K_V \geq 0$ there exists a time $T' = T'(G, \delta, K_V) > 0$ such that for all
 $T \in (0,T']$, all $(G,\delta)$-equidistributed sequences, all measurable and bounded $V\colon \RR^d \to \RR^d$ satisfying
 $\lVert V \rVert_\infty \leq K_V$, all $L \in G \NN$, and any initial datum $u_0\in L^2(\Lambda_L)$ there exists a null-control
 $f \in L^2((0,T),L^2(S_\delta(L)))$ for the initial value problem
 \begin{equation*}
 \partial_t u(t) +H_{\Lambda_L}(0,V) u(t) = f(t)
 \quad\text{ for }\quad 0<t<T,\quad u(0)=u_0,
 \end{equation*}
 with
 \begin{equation*}
  \Vert f\Vert_{L^2((0,T),L^2(S_\delta(L)))}
  \leq
  C_{con}  \cdot\norm{u_0}_{L^2(\Lambda_L)},
 \end{equation*}
 where
 \begin{equation*}
  C_{con}:= 2 \ee^{\lVert V \rVert_\infty} \left(\frac{G}{\delta}\right)^{N ( 1 + G^{4/3} \lVert V \rVert_\infty^{2/3})/2}
  \ee^{c_{\ast} /T},
 \end{equation*}
 $c_{\ast} = \ln (G/\delta)^2  \left( NG + 4 / \ln 2 \right)^2$, {and} $N = N (d)$ is the constant from
 Proposition~\ref{thm:NTTVresult}.
\end{proposition}

Our main result again allows to infer the controllability of the associated inhomogeneous equation on the whole of $\RR^d$.

\begin{corollary}\label{cor:NTTV}
 Let $S=S_\delta\subset\R^d$ be a $(G,\delta)$-equidistributed set, $\tilde u\in L^2(\R^d)$, and $(L_n)_n$ a sequence in $G\NN$ with
 $L_n\nearrow\infty$ as $n\to\infty$, $A\equiv0$, $V\in L^\infty(\RR^d)$, and $T \in (0,T']$ with $T'$ as in
 Proposition~\ref{thm:contcost}. Then, there is a subsequence $(L_{n_k})_k$ of $(L_n)_n$ and null-controls $f_{n_k}$ for the initial
 value problem~\eqref{eq:abstrCauchyH} on $\Omega=\Lambda_{L_{n_k}}$ with $\omega=\Lambda_{L_{n_k}}\cap S$ and
 $u_0=\tilde u|_{\Lambda_{L_{n_k}}}$ with corresponding mild solutions $u_{n_k}$ such that:
 \begin{enumerate}
  \renewcommand{\theenumi}{\roman{enumi}}
  \item $(f_{n_k})_k$ converges weakly in $L^2((0,T),L^2(\R^d))$ to a null-control for the initial value
        problem~\eqref{eq:abstrCauchyH} on $\Omega=\R^d$ with $\omega=S$ and $u_0=\tilde u$;
  \item the corresponding mild solution to~\eqref{eq:abstrCauchyH} on $\Omega=\R^d$ is the weak limit of $(u_{n_k})_k$ in
        $L^2((0,T),L^2(\R^d))$.
 \end{enumerate}
 In particular, the system~\eqref{eq:abstrCauchyH} on $\Omega=\R^d$ with $\omega=S$ is null-controllable, and the associated control
 cost $C_T$ satisfies
 \begin{equation*}
  C_T \le C_{con}
 \end{equation*}
 with $C_{con}$ as in Proposition~\ref{thm:contcost}.
\end{corollary}
This recovers a result of \cite{NakicTTV-pre}. There actually a better estimate on
the control cost is provided.

\begin{remark}
 If, for general $A\in L_\text{loc}^2(\R^d,\R^d)$, we assume that $V\geq 0$, then the semigroup associated with
 $H_{\Lambda_L}(A, V)$ is contractive. In this case, abstract arguments as in~\cite{TenenbaumT-07} or~\cite{BeauchardPS-18} allow to
 reduce the question of null-controllability to the proof of a spectral inequality or uncertainty relation. Thus, if we find a set
 $S$ such that the restriction to $\Lambda_L\cap S$ allows for a scale-independent uncertainty relation of the type~\eqref{eq:sfuc},
 then it automatically follows that the system~\eqref{eq:abstrCauchyH} on $\Omega=\R^d$ with $\omega=S$ is null-controllable.
\end{remark}

\section{Approximation of Schr\"odinger semigroups}\label{sec:semigroups}

In the situation of Subsection~\ref{subsec:main}, denote $H=H_\Gamma(A,V)$ and $H_L=H_{\Gamma\cap\Lambda_L}(A,V)$. Since these
operators are defined on different Hilbert spaces, namely $L^2(\Gamma)$ and $L^2(\Gamma\cap\Lambda_L)$, respectively, we need a
notion of extension in order to compare the associated semigroups. To this end, we identify $H_L$ with the direct sum $H_L\oplus0$
on $L^2(\Gamma)$ with respect to the orthogonal decomposition
$L^2(\Gamma)=L^2(\Gamma\cap\Lambda_L)\oplus L^2(\Gamma\setminus\Lambda_L)$. This agrees with the understanding that every function
in $L^2(\Gamma\cap\Lambda_L)$ is trivially extended to a function in $L^2(\Gamma)$. In this sense, the subspace
$L^2(\Gamma\cap\Lambda_L)\subset L^2(\Gamma)$ is a reducing subspace for the self-adjoint operator $H_L$ on $L^2(\Gamma)$
(cf.~\cite[Definition~1.8]{Schm12}) and therefore also for the exponential $\ee^{-tH_L}=\ee^{-tH_L}\oplus I$ for all $t\ge 0$, see,
e.g., \cite[Satz~8.23]{Wei00}. In particular, $\ee^{-tH_L}$ is a bounded self-adjoint operator on the whole of $L^2(\Gamma)$, and
one has $\ee^{-tH_L}f=0$ on $\Gamma\setminus\Lambda_L$ for all $f\in L^2(\Gamma\cap\Lambda_L)$.

\begin{lemma}\label{lem:semigroup}
 Let $R>0$, $u_0\in L^2(\Gamma\cap\Lambda_R)\subset L^2(\Gamma)$, and $t>0$. Then, there exists a constant $C=C(t,d,V_-)>0$ such
 that for every $L\ge 2R$ one has:
 \begin{enumerate}
  \renewcommand{\theenumi}{\alph{enumi}}
  \item $\norm{(\ee^{-tH}-\ee^{-tH_L})u_0}_{L^2(\Gamma\cap\Lambda_L)}^2
        \le C\exp\bigl(-\frac{L^2}{32t}\bigr)\norm{u_0}_{L^2(\Gamma)}^2$;
  \item $\norm{\ee^{-tH}u_0}_{L^2(\Gamma\setminus\Lambda_L)}^2
        \le C\exp\bigl(-\frac{L^2}{32t}\bigr)\norm{u_0}_{L^2(\Gamma)}^2$;
  \item $\norm{(\ee^{-tH}-\ee^{-tH_L})u_0}_{L^2(\Gamma)}^2
        \le 2C\exp\bigl(-\frac{L^2}{32t}\bigr)\norm{u_0}_{L^2(\Gamma)}^2$.
 \end{enumerate}

 \begin{proof}
  By a standard density argument, it clearly suffices to consider the particular case of
  $u_0\in L^4(\Gamma\cap\Lambda_R)\subset L^2(\Gamma\cap\Lambda_R)$, which we assume throughout the proof.

  Important steps of the following argument are inspired by the poof of Theorem~1 in~\cite{HKNSV06}.

  We denote by $b=(b_t)_{t\ge 0}$ the standard Brownian motion in $\R^d$. Moreover, given $x\in\Gamma$, $\E_x$ stands for the
  expectation with respect to the associated Wiener measure $\P_x$ in $x$.

  Let $L\ge2R$. Then, the Feynman-Kac-It\^o formula in the version from~\cite[Korollar~3.3]{H96} states that for almost every
  $x\in\Gamma$ one has
  \begin{equation*}
   (\ee^{-tH}u_0)(x)=\E_x\bigl\{u_0(b_t)\ee^{-\ii S_t(A)-\int_0^t V(b_s)\,\dd s}
    \cdot\one_{\{\forall s\in[0,t]\,:\,b_s\in\Gamma\}}(b)\bigr\},
  \end{equation*}
  where $S_t(A)$ is a real-valued stochastic process. A variant of this formula under stronger regularity assumptions on the vector
  potential $A$ can also be found in~\cite[Proposition~2.3]{BHL00}; see also the references therein.

  In case of $x\in\Gamma\cap\Lambda_L$, the analogous formula for $H_L$ reads
  \begin{equation*}
   (\ee^{-tH_L}u_0)(x)=\E_x\bigl\{u_0(b_t)\ee^{-\ii S_t(A)-\int_0^t V(b_s)\,\dd s}
    \cdot\one_{\{\forall s\in[0,t]\,:\,b_s\in\Gamma\cap\Lambda_L\}}(b)\bigr\}.
  \end{equation*}
  Consequently, in view of $u_0=0$ on $\Gamma\setminus\Lambda_R$, for almost every $x\in\Gamma\cap\Lambda_L$ we have
  \begin{equation*}
   \bigl((\ee^{-tH}-\ee^{-tH_L})u_0\bigr)(x)=
   \E_x\bigl\{u_0(b_t)\ee^{-\ii S_t(A)-\int_0^t V(b_s)\,\dd s}\Phi_{\Gamma,L,t}(b)\bigr\}
  \end{equation*}
  with $\Phi_{\Gamma,L,t}(b):=\one_{\{\forall s\in[0,t]\,:\,b_s\in\Gamma,\,b_t\in\Lambda_R,\,
  \exists s\in[0,t]\,:\,b_s\notin\Lambda_L\}}(b)$. Thus, setting
  \begin{equation*}
   \Phi_{L,t}(b):=\one_{\{\,b_t\in\Lambda_R,\,\exists s\in[0,t]\,:\,b_s\notin\Lambda_L\}}(b)\ge\Phi_{\Gamma,L,t}(b),
  \end{equation*}
  this gives
  \begin{equation}\label{eq:semigroupDiff}
   \begin{aligned}
    \abs{\bigl((\ee^{-tH}-\ee^{-tH_L})u_0\bigr)(x)}
    &\le \E_x\bigl\{\abs{u_0(b_t)} \ee^{\int_0^t V_-(b_s)\,\dd s}\Phi_{L,t}(b)\bigr\},
   \end{aligned}
  \end{equation}
  where we have taken into account that $S_t(A)$ is real valued. Similarly, for almost every $x\in\Gamma\setminus\Lambda_L$ we
  obtain that
  \begin{equation}\label{eq:semigroupPlain}
   \begin{aligned}
    \abs{(\ee^{-tH}u_0)(x)}
    &\le \E_x\bigl\{\abs{u_0(b_t)}\ee^{\int_0^t V_-(b_s)\,\dd s}\cdot\one_{\{b_t\in\Lambda_R\}}(b) \bigr\}\\
    &= \E_x\bigl\{\abs{u_0(b_t)}\ee^{\int_0^t V_-(b_s)\,\dd s}\Phi_{L,t}(b) \bigr\},
   \end{aligned}
  \end{equation}
  where for the last equality we have taken into account that $\P_x$-almost surely $b_0=x\notin\Lambda_L$.

  Extending $u_0$ trivially to the whole of $\R^d$, the assumption $u_0\in L^4(\Gamma\cap\Lambda_R)$ implies that
  $\abs{u_0}^2\in L^2(\R^d)$. Hence,
  \begin{equation*}
   \E_x\bigl\{\abs{u_0(b_t)}^2\ee^{2\int_0^t V_-(b_s)\,\dd s}\bigr\}=\bigl(\ee^{-tH_{\R^d}(0,-2V_-)}\abs{u_0}^2\bigr)(x)
  \end{equation*}
  by the standard Feynman-Kac formula on $\R^d$, see, e.g.,~\cite[Theorem~A.2.7]{Sim82}. The Cauchy-Schwarz inequality with respect
  to $\E_x$ therefore yields
  \begin{equation}\label{eq:expect}
   \begin{aligned}
    \E_x\{\abs{&u_0(b_t)}\ee^{\int_0^t V_-(b_s)\,\dd s}\Phi_{L,t}(b)\}^2\\
    &\le\bigl(\ee^{-tH_{\R^d}(0,-2V_-)}\abs{u_0}^2\bigr)(x)\cdot\E_x\{\Phi_{L,t}^2(b)\}\\
    &=\bigl(\ee^{-tH_{\R^d}(0,-2V_-)}\abs{u_0}^2\bigr)(x)\cdot\P_x\{b_t\in\Lambda_R,\,\exists s\in[0,t]:b_s\notin\Lambda_L\}.
   \end{aligned}
  \end{equation}

  Thus, in view of~\eqref{eq:semigroupDiff},~\eqref{eq:semigroupPlain}, and~\eqref{eq:expect}, it remains to estimate the
  probability $\P_x\{b_t\in\Lambda_R,\,\exists s\in[0,t]:b_s\notin\Lambda_L\}$ for almost every $x\in\Gamma$. In order to do so,
  observe that $b_t\in\Lambda_R$ and $b_s\notin\Lambda_L$ for some $s\in[0,t]$ can only happen if for some coordinate
  $j\in\{1,\dots,d\}$ the one-dimensional Brownian motion $(b_s^j)_{s\ge0}$ obtained as the $j$-coordinate of $(b_s)_{s\ge0}$
  satisfies
  \begin{equation*}
   \abs{b_t^j}<\frac{R}{2} \quad\text{ and }\quad \max_{0\le s\le t}\abs{b_s^j}\ge\frac{L}{2}.
  \end{equation*}
  In particular, this requires that
  \begin{equation*}
   \max_{0\le s\le t} \abs{b_t^j-b_s^j} \ge \frac{L-R}{2}\ge \frac{L}{4}.
  \end{equation*}
  The probability for this to happen can for each coordinate $j$ be estimated as
  \begin{align*}
   \P_x\Bigl\{\max_{0\le s\le t}\abs{b_t^j-b_s^j}\ge\frac{L}{4}\Bigr\}
   &\le 2\P_0\Bigl\{\max_{0\le s\le t}(b_t^j-b_s^j)\ge\frac{L}{4}\Bigr\}\\
   &= 2\P_0\Bigl\{\max_{0\le s\le t}b_{t-s}^j\ge\frac{L}{4}\Bigr\}
   = 2\P_0\Bigl\{\max_{0\le s\le t}b_s^j\ge\frac{L}{4t}t\Bigr\}\\
   &\le 2\exp\Bigl(-\Bigl(\frac{L}{4t}\Bigr)^2\cdot \frac{t}{2}\Bigr)=2\exp\Bigl(-\frac{L^2}{32t}\Bigr),
  \end{align*}
  where for the last line we applied a standard exponential inequality for the one-dimensional Brownian motion, see,
  e.g.,~\cite[Satz~46.5]{Bau91}. This yields
  \begin{equation}\label{eq:prob}
   \P_x\{b_t\in\Lambda_R,\,\exists s\in[0,t]:b_s\notin\Lambda_L\} \le 2d\exp\Bigl(-\frac{L^2}{32t}\Bigr).
  \end{equation}

  Recall (see, e.g.,~\cite[Theorem~2.1]{CFKS87}) that $\ee^{-tH_{\R^d}(0,-2V_-)}$ can be extended to a bounded operator from
  $L^1(\R^d)$ to $L^1(\R^d)$, that is,
  \begin{equation*}
   c(t,d,V_-):=\norm{\ee^{-tH_{\R^d}(0,-2V_-)}}_{L^1(\R^d)\to L^1(\R^d)}<\infty.
  \end{equation*}
  Hence, by combining~\eqref{eq:semigroupDiff},~\eqref{eq:expect}, and~\eqref{eq:prob} we obtain
  \begin{equation*}
   \abs{\bigl((\ee^{-tH}-\ee^{-tH_L})u_0\bigr)(x)}^2\le 2d\exp\Bigl(-\frac{L^2}{32t}\Bigr)
   \bigl(\ee^{-tH_{\R^d}(0,-2V_-)}\abs{u_0}^2\bigr)(x)
  \end{equation*}
  for almost every $x\in\Gamma\cap\Lambda_L$, and integrating over $\Gamma\cap\Lambda_L$ results in
  \begin{align*}
   \norm{(\ee^{-tH}-\ee^{-tH_L})u_0}_{L^2(\Gamma\cap\Lambda_L)}^2
   &\le 2d\exp\Bigl(-\frac{L^2}{32t}\Bigr) \norm{\ee^{-tH_{\R^d}(0,-2V_-)}\abs{u_0}^2}_{L^1(\R^d)}\\
   &\le 2dc(t,d,V_-)\exp\Bigl(-\frac{L^2}{32t}\Bigr) \norm{\abs{u_0}^2}_{L^1(\R^d)}\\
   &= 2dc(t,d,V_-)\exp\Bigl(-\frac{L^2}{32t}\Bigr) \norm{u_0}_{L^2(\Gamma)}^2.
  \end{align*}
  This proves part (a) with $C(t,d,V_-)=2dc(t,d,V_-)$.

  Part (b) is proved in an analogous way by combining~\eqref{eq:semigroupPlain},~\eqref{eq:expect}, and~\eqref{eq:prob} for
  $x\in\Gamma\setminus\Lambda_L$ and integrating over $\Gamma\setminus\Lambda_L$.

  Finally, taking into account that $\ee^{-tH_L}u_0=0$ on $\Gamma\setminus\Lambda_L$, part (c) is a direct consequence of (a) and
  (b) and Pythagoras' identity. This completes the proof of the lemma.
 \end{proof}
\end{lemma}

\begin{remark}\label{rem:const}
 One can show that there are constants $c_0,c_1>0$, depending only on $d$ and $V_-$, such that
 \begin{equation*}
  c(t,d,V_-)=\norm{\ee^{-tH_{\R^d}(0,-2V_-)}}_{L^1(\R^d)\to L^1(\R^d)}\le c_0\ee^{tc_1}
 \end{equation*}
 and, in turn,
 \begin{equation*}
  C(t,d,V_-)=2d\norm{\ee^{-tH_{\R^d}(0,-2V_-)}}_{L^1(\R^d)\to L^1(\R^d)}\le 2dc_0\ee^{tc_1}.
 \end{equation*}
 Indeed, it follows from~\cite[Theorem~4.7]{AS82} and a duality argument that
 \begin{equation*}
  \lim_{t\searrow0}\norm{\ee^{-tH_{\R^d}(0,-2V_-)}}_{L^1(\R^d)\to L^1(\R^d)}=1.
 \end{equation*}
 Hence, we see from the semigroup property that the constants $c_0,c_1$ can be chosen as
 $c_0:=\sup_{0\le s\le \delta}\norm{\ee^{-sH_{\R^d}(0,-2V_-)}}_{L^1(\R^d)\to L^1(\R^d)}<\infty$ and
 \begin{equation*}
  c_1:=\frac{1}{\delta}\ln\norm{\ee^{-\delta H_{\R^d}(0,-2V_-)}}_{L^1(\R^d)\to L^1(\R^d)}
 \end{equation*}
 for some $\delta>0$.

 Note that in case of $V_-=0$ one has $c(t,d,0)\le 1$ and, in turn,
 \begin{equation*}
  C(t,d,0)\le 2d
 \end{equation*}
 since $-H_{\R^d}(0,0)=\Delta_{\R^d}$ is known to generate a contraction semigroup on $L^1(\R^d)$, see,
 e.g.,~\cite[Theorems~1.4.1 and~1.3.5]{Dav89}.
\end{remark}

\begin{remark}\label{rem:genExhaust}
 The explicit form of dependence on the length scale $L$ in the estimates of Lemma~\ref{lem:semigroup} is due to the specific choice
 of the exhaustion $(\Lambda_L)_{L>0}$ of $\R^d$. If one is interested in more qualitative statements only, also more general
 exhaustions $(\Omega_n)_{n\in\N}$ of $\R^d$ can be considered: Let $\Omega_n\subset\R^d$ be open with $\Omega_n\subset\Omega_{n+1}$
 for all $n\in\N$, $\bigcup_{n\in\N}\Omega_n=\R^d$, and suppose that for every $k\in\N$ one has
 \begin{equation*}
  d_k(n):=\dist(\Omega_k,\R^d\setminus\Omega_n)\to \infty\quad\text{ as }\quad n\to\infty.
 \end{equation*}
 Under the assumption $u_0\in L^4(\Gamma\cap\Omega_k)$ for some $k\in\N$, the condition $b_t\in\Omega_k$ and $b_s\notin\Omega_n$ for
 some $s\in[0,t]$ from the proof of Lemma~\ref{lem:semigroup} (with $\Lambda$ replaced by $\Omega$) then requires that
 $\max_{0\le s\le t}\abs{b_t-b_s}\ge d_k(n)$ and, thus,
 \begin{equation*}
  \max_{0\le s\le t}\abs{b_t^j-b_s^j}\ge\frac{d_k(n)}{\sqrt{d}}
 \end{equation*}
 for some coordinate $j$. The rest of the reasoning stays exactly the same, but with the term $\exp(-L^2/(32t))$ replaced by
 $\exp(-d_k^2(n)/(2dt))$.
\end{remark}

Lemma~\ref{lem:semigroup} in conjunction with Remark~\ref{rem:genExhaust} yields the following result.

\begin{corollary}\label{cor:semigroup}
 Let $(\Omega_n)_n$ be an exhaustion of $\R^d$ as in Remark~\ref{rem:genExhaust}. Then, the self-adjoint operators $H=H_\Gamma(A,V)$
 and $H_n=H_{\Gamma\cap\Omega_n}(A,V)$ have a common lower bound independent of $n$, and one has $\ee^{-tH_n}\to\ee^{-tH}$ strongly
 in $L^2(\Gamma)$ as $n\to\infty$ for all $t>0$.

 \begin{proof}
  That $H$ and $H_n$ have a common lower bound independent of $n$ is an immediate consequence of the fact that the quadratic form
  for $H$ extends the one for $H_n$. Let $a\in\R$ be such a common lower bound.

  Let $t>0$, $g\in L^2(\Gamma)$, $\eps>0$, and choose $k\in\N$ with
  \begin{equation*}
   \norm{g-g\chi_{\Gamma\cap\Omega_k}}_{L^2(\Gamma)}<\eps\ee^{ta}/4,
  \end{equation*}
  where $\chi_{\Gamma\cap\Omega_k}\colon\Gamma\to\R$ denotes the characteristic function for $\Gamma\cap\Omega_k$. It then follows
  from part (c) of Lemma~\ref{lem:semigroup} and Remark~\ref{rem:genExhaust} that for some $n_0>k$ and all $n\ge n_0$ one has
  \begin{equation*}
   \norm{(\ee^{-tH}-\ee^{-tH_n})g\chi_{\Gamma\cap\Omega_k}}_{L^2(\Gamma)}<\frac{\eps}{2}
  \end{equation*}
  and, therefore,
  \begin{align*}
   \norm{(\ee^{-tH}-\ee^{-tH_n})g}_{L^2(\Gamma)}
   &\le \norm{(\ee^{-tH}-\ee^{-tH_n})(g-g\chi_{\Gamma\cap\Omega_k})}_{L^2(\Gamma)} + \frac{\eps}{2}\\
   &\le 2\ee^{-ta}\norm{g-g\chi_{\Gamma\cap\Omega_k}}_{L^2(\Gamma)} + \frac{\eps}{2} < \eps.
  \end{align*}
  This completes the proof.
 \end{proof}%
\end{corollary}

\begin{remark}\label{rem:Dan}
 In the particular case of $A=0$ and $V=0$, Corollary~\ref{cor:semigroup} can also be obtained in a more abstract way without the
 use of Lemma~\ref{lem:semigroup}:

 By~\cite[Proposition~7.1 and Theorem~3.3]{Dan03}, for every $\lambda$ with $\Re\lambda>0$ the resolvents
 $(\lambda-H_{\Gamma\cap\Omega_n}(0,0))^{-1}$ converge strongly to $(\lambda-H_\Gamma(0,0))^{-1}$. The claim of the corollary then
 follows from classic results on strongly continuous semigroups, see, e.g.,~\cite[Theorem~3.42]{Pazy83}.

 However, for the general case discussed in Corollary~\ref{cor:semigroup} we have not found a reference in the literature.
\end{remark}

\section{Continuous dependence on inhomogeneity and proof of Theorem~\ref{thm:main}}\label{sec:abstrCauchy}

For this whole section, we introduce the following setting as a general framework.
\begin{hypothesis}\label{hyp:controlConv}
 Let $\cH$ and $\cU$ be Hilbert spaces, and let $H,H_n$, $n\in\N$, be lower semibounded self-adjoint operators on $\cH$ with a
 common lower bound $a\in\R$
 such that $(\ee^{-tH_n})_n$ converges strongly to $\ee^{-tH}$ for all $t>0$. Moreover, let $B,B_n$, $n\in\N$, be bounded operators
 from $\cU$ to $\cH$ such that $(B_n)_n$ und $(B_n^*)_n$ converge strongly to $B$ and $B^*$, respectively. Let $T>0$, and let
 $\cB^T \colon L^2((0,T),\cU) \to \cH$ be the controllability map associated to the system~\eqref{eq:abstrCauchy}. Finally, let
 $\cB_n^T \colon L^2((0,T),\cU) \to \cH$, $n\in\N$, be the controllability map associated to the corresponding system with $H$ and
 $B$ replaced by $H_n$ and $B_n$, respectively.
\end{hypothesis}

Recall that $L^2((0,T),\cH)$ is a Hilbert space with respect to the inner product
\begin{equation*}
 \langle g,h \rangle_{L^2((0,T),\cH)}=\int_0^T \langle g(t),h(t)\rangle_\cH\,\dd t\,,\quad g,h\in L^2((0,T),\cH).
\end{equation*}
Moreover, given $h\in L^2((0,T),\cH)$ and a measurable subset $I\subset(0,T)$, we understand the integral $\int_I h(t)\,\dd t\in\cH$
in the weak sense, that is, $\int_I h(t)\,\dd t$ denotes by Riesz' theorem the unique element in $\cH$ for which
\begin{equation*}
 \Bigl\langle g,\int_I h(t)\,\dd t\Bigr\rangle_\cH = \int_I\langle g,h(t)\rangle_\cH\,\dd t
 \quad\text{ for all }\quad g\in\cH.
\end{equation*}
In particular, one has
\begin{equation*}
 \norm{\int_I h(t)\,\dd t}_\cH \le \int_I \norm{h(t)}_\cH\,\dd t.
\end{equation*}
The same applies for the Hilbert space $\cU$ instead of $\cH$.

Our first convergence result addresses the strong convergence of the controllability maps and their adjoints.
\begin{lemma}\label{lem:controllabilitymap}
 Assume Hypothesis~\ref{hyp:controlConv}. Then, $(\cB_n^T)_n$ and $((\cB_n^T)^*)_n$ converge strongly to $\cB^T$ and $(\cB^T)^*$,
 respectively.

 \begin{proof}
  First, observe that $\norm{\ee^{-(T-s)H_n}} \le \ee^{T\abs{a}}$ for all $n$ and $0 < s < T$. Moreover, the sequence $(B_n)_n$ is
  uniformly bounded as a strongly convergent sequence, and $(\ee^{-(T-s)H_n}B_n)_n$ converges strongly to $\ee^{-(T-s)H}B$ for
  $0<s<T$. Hence, for every $f\in L^2((0,T),\cU)$ we conclude that
  \begin{equation*}
   \norm{ \cB_n^T f - \cB^T f}_\cH \le \int_0^T \norm{ \ee^{-(T-s)H_n}B_n f(s) - \ee^{-(T-s)H}B f(s) }_\cH \,\dd s
   \xrightarrow[n\to\infty]{}0
  \end{equation*}
  by Lebegue's dominated convergence theorem. This proves the claim for $(\cB_n^T)_n$.

  In order to show the claim for $((\cB_n^T)^*)_n$, we observe that
  \begin{equation*}
   (\cB^T)^* = B^* \ee^{-(T-\,\cdot)H} \quad\text{ and }\quad
   (\cB_n^T)^* = B_n^* \ee^{-(T-\,\cdot)H_n}
  \end{equation*}
  for every $n$. Indeed, for $f\in L^2((0,T),\cU)$ and $g\in\cH$ we compute
  \begin{align*}
   \langle g,\cB^T f \rangle_\cH
   &= \int_0^T \langle g,\ee^{-(T-s)H}B f(s) \rangle_\cH\,\dd s
      = \int_0^T \langle B^*\ee^{-(T-s)H} g, f(s) \rangle_\cU\,\dd s\\
   &= \langle B^*\ee^{-(T-\,\cdot)H} g, f \rangle_{L^2((0,T),\cU)}
  \end{align*}
  and analogoulsy for $\cB_n^T$. Now, the claim for $((\cB_n^T)^*)_n$ follows in the same way as above.
 \end{proof}%
\end{lemma}

We are now able to prove the main result of this section.

\begin{lemma}\label{lem:weakStrongConv}
 Assume Hypothesis~\ref{hyp:controlConv}. Moreover, let $(u_{0,n})_n$ be a sequence in $\cH$ converging in norm to some $u_0\in\cH$.
 Let $f,f_n\in L^2((0,T),\cU)$, $n\in\N$. Denote by $u$ and $u_n$, $n\in\N$, the mild solutions to the abstract Cauchy problems
 \begin{equation}\label{eq:abstrCauchy_limit}
  \partial_t u(t) + H u(t) = B f(t) \quad\text{ for }\quad 0<t<T,\quad u(0)=u_0,
 \end{equation}
 and
 \begin{equation}\label{eq:abstrCauchy_n}
  \partial_t u_n(t) + H_n u_n(t) = B_n f_n(t) \quad\text{ for }\quad 0<t<T,\quad u_n(0)=u_{0,n},
 \end{equation}
 respectively.
 \begin{enumerate}
  \item If $(f_n)_n$ converges to $f$ in $L^2((0,T),\cU)$, then $(u_n(t))_n$ converges to $u(t)$ in $\cH$ for all $t\in(0,T]$.
        Moreover, $(u_n)_n$ converges to $u$ in $L^2((0,T),\cH)$.
  \item If $(f_n)_n$ converges to $f$ weakly in $L^2((0,T),\cU)$, then $(u_n(t))_n$ converges to $u(t)$ weakly in $\cH$ for all
        $t\in(0,T]$. Moreover, $(u_n)_n$ converges to $u$ weakly in $L^2((0,T),\cH)$.
 \end{enumerate}

 \begin{proof}
  By definition of the mild solution and the controllability map, we have
  \begin{equation}\label{eq:sep:u_n}
   u_n(t) = \ee^{-tH_n} u_{0,n} + \int_0^t \ee^{-(t-s)H_n} B_n f_n(s)\,\dd s
   = \ee^{-tH_n} u_{0,n} + \cB_n^t f_n|_{(0,t)}
  \end{equation}
  and, analogously,
  \begin{equation}\label{eq:sep:u}
   u(t) = \ee^{-tH} u_0 + \cB^t f|_{(0,t)}.
  \end{equation}
  Moreover, observe that
  \begin{equation}\label{eq:cBntBound}
   \norm{ \cB_n^t } \le \sqrt{t}\ee^{t\abs{a}}\norm{B_n} \le \sqrt{T}\ee^{T\abs{a}}\norm{B_n}.
  \end{equation}

  Now, if $(f_n)_n$ converges to $f$ in $L^2((0,T),\cU)$, applying Lemma~\ref{lem:controllabilitymap} to $\cB_n^t$ and $\cB^t$
  implies that $(u_n(t))_n$ converges to $u(t)$ in $\cH$ for all $t\in (0,T]$.

  Taking into account that by~\eqref{eq:sep:u_n} and~\eqref{eq:cBntBound} one has
  \begin{equation}\label{eq:mildSolBound}
   \begin{aligned}
    \norm{ u_n(t) }_\cH
    &\le \ee^{-ta}\norm{ u_{0,n} }_\cH + \norm{ \cB_n^t }\norm{f_n}_{L^2((0,t),\cU)}\\
    &\le \ee^{T\abs{a}}\norm{ u_{0,n} }_\cH + \sqrt{T}\ee^{T\abs{a}}\norm{B_n}\norm{f_n}_{L^2((0,T),\cU)}
   \end{aligned}
  \end{equation}
  for all $n$, the boundedness of the sequences $(u_{0,n})_n$, $(f_n)_n$, and $(B_n)_n$ yields that
  \begin{equation*}
   \norm{ u - u_n }_{L^2((0,T),\cH)}^2 = \int_0^T \norm{ u(t) - u_n(t) }_\cH^2 \,\dd t \xrightarrow[n\to\infty]{}0
  \end{equation*}
  by Lebesgue's dominated convergence theorem. This completes the proof of part~(a).

  Now suppose that $(f_n)_n$ converges to $f$ only weakly in $L^2((0,T),\cU)$. Since then $(f_n)_n$ is still bounded, applying
  Lemma~\ref{lem:controllabilitymap} to $(\cB_n^t)^*$ and $(\cB^t)^*$ implies that for every $g\in\cH$ one has
  \begin{align*}
   \langle g,\cB_n^t f_n|_{(0,t)} \rangle_\cH
   &= \langle (\cB_n^t)^* g, f_n \rangle_{L^2((0,t),\cU)}\\
   &\xrightarrow[n\to\infty]{}
   \langle (\cB^t)^* g, f \rangle_{L^2((0,t),\cU)} = \langle g,\cB^t f|_{(0,t)} \rangle_\cH,
  \end{align*}
  that is, $(\cB_n^t f_n|_{(0,t)})_n$ converges to $\cB^t f|_{(0,t)}$ weakly in $\cH$. In turn, by~\eqref{eq:sep:u_n}
  and~\eqref{eq:sep:u}, $(u_n(t))_n$ converges to $u(t)$ weakly in $\cH$ for all $t\in(0,T]$.

  Taking into account~\eqref{eq:mildSolBound} and the boundedness of $(u_{0,n})_n$, $(f_n)_n$, and $(B_n)_n$, we now conclude for
  every $h\in L^2((0,T),\cH)$ that
  \begin{align*}
   \langle h,u_n\rangle_{L^2((0,T),\cH)}
   &= \int_0^T \langle h(t),u_n(t)\rangle_\cH\,\dd t\\
   &\xrightarrow[n\to\infty]{}\int_0^T \langle h(t),u(t)\rangle_\cH\,\dd t=\langle h,u\rangle_{L^2((0,T),\cH)}
  \end{align*}
  by Lebesgue's dominated convergence theorem. This shows~(b) and, hence, completes the proof.
 \end{proof}%
\end{lemma}

\begin{remark}
 The moral of the proof of Lemma~\ref{lem:weakStrongConv} is as follows: the (weak) convergence of $(u_n(T))_n$ to $u(T)$ in $\cH$
 follows easily from Lemma~\ref{lem:controllabilitymap}. From this, we get the (weak) convergence of $(u_n(t))_n$ to $u(t)$ in $\cH$
 for every $t\in(0,T]$ by replacing $T$ with $t$ and taking into account that $(f_n|_{(0,t)})_n$ converges to $f|_{(0,t)}$(weakly)
 in $L^2((0,t),\cU)$. Finally, the (weak) convergence of $(u_n)_n$ to $u$ in $L^2((0,T),\cH)$ follows by Lebesgue's dominated
 convergence theorem.
\end{remark}

\begin{corollary}\label{cor:weakConv}
 If in the situation of Lemma~\ref{lem:weakStrongConv} the sequence $(f_n)_n$ converges to $f$ weakly in $L^2((0,T),\cH)$ and each
 $f_n$ is a null-control for the initial value problem~\eqref{eq:abstrCauchy_n}, then $f$ is a null-control for the initial value
 problem~\eqref{eq:abstrCauchy_limit}.

 \begin{proof}
  It follows from the first statement in part~(b) of Lemma~\ref{lem:weakStrongConv} for $t=T$ that
  \begin{equation*}
   \norm{u(T)}_\cH^2 = \langle u(T),u(T)\rangle_\cH = \lim_{n\to\infty} \langle u(T),u_n(T)\rangle_\cH = 0,
  \end{equation*}
  that is, $u(T)=0$, which proves the claim.
 \end{proof}%
\end{corollary}

\begin{remark}\label{rem:weakConv}
 As follows from the proof of Lemma~\ref{lem:weakStrongConv}, the statements of Lemma~\ref{lem:weakStrongConv}\,(b) and
 Corollary~\ref{cor:weakConv} are still valid if the sequence $(u_{0,n})_n$ of initial data converges in $\cH$ only weakly. Indeed,
 in this case, one has
 \begin{equation*}
  \langle g,\ee^{-tH_n}u_{0,n}\rangle_\cH = \langle \ee^{-tH_n}g,u_{0,n}\rangle_\cH\xrightarrow[n\to\infty]{}
  \langle \ee^{-tH}g,u_{0}\rangle_\cH=\langle g,\ee^{-tH}u_{0}\rangle_\cH
 \end{equation*}
 for every $g\in\cH$, so that the first summand on the right-hand side of~\eqref{eq:sep:u_n} converges weakly to $\ee^{-tH}u_0$. The
 rest of the reasoning then stays exactly the same.
\end{remark}

\begin{proof}[Proof of Theorem \ref{thm:main}]
 In the situation of Lemma~\ref{lem:weakStrongConv}, take $\cH=\cU=L^2(\Gamma)$, $B=\chi_{\Gamma\cap S}$,
 $B_n=\chi_{\Gamma_{L_n}\cap S}$, $H=H_\Gamma(A,V)$, $H_n=H_{L_n}(A,V)$, $u_0=\tilde{u}$, and
 $u_{0,n}=\chi_{\Gamma_{L_n}}\cdot\tilde{u}$. Then, one has $B_n^*=B_n$, and $(B_n)_n$ converges strongly to $B=B^*$. Thus, in view
 of Corollary~\ref{cor:semigroup}, Hypothesis~\ref{hyp:controlConv} is satisfied. Moreover, $(u_{0,n})_n$ converges to $u_0$ in
 $L^2(\Gamma)$. The claim of the theorem now follows immediately from part~(b) of Lemma~\ref{lem:weakStrongConv} and
 Corollary~\ref{cor:weakConv}.
\end{proof}

\begin{remark}[An alternative approach]\label{rem:alternative}
We sketch an alternative strategy to obtain null-controllability of the limiting system from null-controllability of the
corresponding system on exhausting domains:
Assume Hypothesis~\ref{hyp:controlConv}, and, in addition, let $P_n$, $n\in\N$, be orthogonal projections in $\cH$ such that
$(P_n)_n$ converges strongly to the identity operator on $\cH$.

If there is a constant $C>0$ such that for all $n \in \N$ and all $v_0 \in \cH$ one has
\begin{equation}\label{eq:observabilityHn}
 \norm{\ee^{-TH_n}P_nv_0}_\cH^2 \le C^2 \int_0^T \norm{ B^*_n\ee^{-tH_n}P_nv_0 }_\cU^2\,\dd t,
\end{equation}
then we deduce by taking the limit as $n\to\infty$, arguing as in the proof of Lemma~\ref{lem:controllabilitymap}, that
\begin{equation}\label{eq:observabilityH}
 \norm{ \ee^{-TH}v_0 }_\cH^2 \le C^2 \int_0^T \norm{ B^*\ee^{-tH}v_0 }_\cU^2 \,\dd t
\end{equation}
for all $v_0\in\cH$. The latter means that the homogeneous Cauchy problem
\begin{equation}\label{eq:homogeneousCauchy}
 \partial_t v(t) + H v(t) = 0 \quad\text{ for }\quad 0<t<T,\quad v(0)=v_0,
\end{equation}
satisfies a so-called final-state-observability inequality, where the term `final-state' refers to  the state $v(T)= \ee^{-TH}v_0$
at time $T$, and the observation is determined by the adjoint of the control operator $B$. We do not discuss in detail the notion of
observability but refer the reader to~\cite{Coron07},~\cite{TW09}, or~\cite{EgidiNSTTV}; see also Remark~\ref{rem:obs} in
Appendix~\ref{sec:optimalFeedback} below. The main point is that, under the given assumptions, null-controllability for the
inhomogeneous Cauchy problem~\eqref{eq:abstrCauchy_limit} is equivalent to final-state-observability of the corresponding
homogeneous Cauchy problem~\eqref{eq:homogeneousCauchy}, and the associated control cost agrees with the minimal possible constant
$C$ in~\eqref{eq:observabilityH}.

In the situation of Subsection~\ref{subsec:main}, in addition to the choices made in the proof of Theorem~\ref{thm:main}, we take
$P_n=\chi_{\Gamma_{L_n}}$. Then, null-controllability of the system~\eqref{eq:abstrCauchyH} on $\Omega=\Gamma_{L_n}$ with a uniform
bound $C$ on the associated control cost, as established in the particular situations discussed in Section~\ref{sec:applications},
leads to an inequality of the type~\eqref{eq:observabilityHn}. The above argument then shows that this inequality carries over to
the limiting domain, yielding null-controllability along with a corresponding bound on the associated control cost.
\end{remark}

\subsection*{Convergence of feedback operators}
Most of the convergence results for null-controls which we provided 
can be lifted to convergence of so-called feedback operators.
To exemplify this, we formulate as a closing of this section a consequence of Corollary~\ref{cor:weakConv}.

We call a mapping $F\colon \cH\to L^2((0,T),\cU)$ a~\emph{feedback operator} associated to the system~\eqref{eq:abstrCauchy_limit}
if
\begin{equation*}
 \ee^{-TH}+\cB^T F=0,
\end{equation*}
that is, if $F$ maps every initial datum $u_0$ to a corresponding null-control.

It is well known (see, e.g., \cite[Proposition~12.1.2]{TW09}; see also Appendix~\ref{sec:optimalFeedback} below) that if the
system~\eqref{eq:abstrCauchy_limit} is null-controllable in time $T>0$, then an associated bounded linear feedback operator exists.

>From Corollary~\ref{cor:weakConv} we obtain the following result.

\begin{corollary}\label{cor:weakFeedback}
 Assume Hypothesis~\ref{hyp:controlConv}, and let $F_n$, $n\in\N$, be a bounded linear feedback operator for the
 system~\eqref{eq:abstrCauchy_n} such that $(F_n)_n$ converges weakly to some $F\colon \cH\to L^2((0,T),\cU)$. Then, $F$ is a
 feedback operator for the system~\eqref{eq:abstrCauchy_limit}.

 \begin{proof}
  Let $u_0\in\cH$, $f:=F u_0$, and $f_n:=F_n u_0$. Then, each $f_n$ is a null-control for the system~\eqref{eq:abstrCauchy_n}, and
  $(f_n)_n$ converges weakly to $f$ by hypothesis. Thus, $f$ is a null-control for the system~\eqref{eq:abstrCauchy_limit} by
  Corollary~\ref{cor:weakConv}, which proves the claim.
 \end{proof}%
\end{corollary}

If $F_n\colon \cH\to L^2((0,T),\cU)$ for each $n\in\N$ is a bounded linear feedback operator for the system~\eqref{eq:abstrCauchy_n}
such that $\sup_n\norm{F_n}<\infty$ and if $\cH$ is separable, then it is well known that $(F_n)_n$ has a weakly convergent
subsequence $(F_{n_k})_k$, see, e.g., \cite[Ex.~4.26]{Wei80}. The corresponding weak limit $F$ satisfies
\begin{equation*}
 \norm{F} \le \liminf \norm{F_{n_k}}.
\end{equation*}
Corollary~\ref{cor:weakFeedback} can now by applied to every such weakly convergent subsequence. Since the optimal feedback operator
$\cF^T$ for the system~\eqref{eq:abstrCauchy_limit} has minimal operator norm, this implies that
\begin{equation*}
 \norm{\cF^T} \le \liminf \norm{F_n}.
\end{equation*}

\appendix

\section{The optimal feedback operator}\label{sec:optimalFeedback}

As in the main part of the paper, let $\cH$ and $\cU$ be Hilbert spaces, $H$ a lower semibounded self-adjoint operator on $\cH$,
$B\colon\cU\to\cH$ a bounded linear operator, and $T>0$. Denote by $\cB^T\colon L^2((0,T),\cU)\to\cH$ the controllability map
associated to the system
\begin{equation}\label{eq:abstrCauchy:app}
 \partial_t u(t) + H u(t) = Bf(t) \quad\text{ for }\quad 0<t<T,\quad u(0)=u_0,
\end{equation}
with $u_0\in\cH$ and $f\in L^2((0,T),\cU)$.

Clearly, the controllability map $\cB^T$ is linear and bounded. In particular, given two null-controls $f$ and $\tilde{f}$ for the
initial datum $u_0$, one has
\begin{equation}\label{eq:nullcontrol}
 0=(\ee^{-TH}u_0 + \cB^T f) - (\ee^{-TH}u_0+\cB^T \tilde{f}) = \cB^T(f-\tilde{f}),
\end{equation}
so that the set of all null-controls for the initial datum $u_0$ is given by
\begin{equation}\label{eq:setofnullcontrols}
 f + \Ker \cB^T
\end{equation}
for any such null-control $f$.

Recall that a mapping $F\colon \cH\to L^2((0,T),\cU)$ is called a~\emph{feedback operator} associated to the
system~\eqref{eq:abstrCauchy:app} if
\begin{equation*}
 \ee^{-TH}+\cB^T F=0,
\end{equation*}
that is, if $F$ maps every initial datum $u_0$ to a corresponding null-control.

The following well-known abstract result from~\cite{TW09} guarantees the existence of bounded linear feedback operators for
null-controllable systems. For convenience of the reader, we give the whole statement, but reproduce only the part of the proof that
we need.

\begin{proposition}[{\cite[Proposition~12.1.2]{TW09}}]\label{prop:TW}
 Let $\cH_1, \cH_2,\cH_3$ be Hilbert spaces, and let $X\colon \cH_1 \to \cH_3$, $Y\colon \cH_2 \to \cH_3$ be bounded linear
 operators. Then, the following are equivalent:
 \begin{enumerate}
  \item $\Ran X \subset \Ran Y$;
  \item There is $c>0$ such that $\norm{X^* z} \le c \norm{Y^* z}$ for all $z\in \cH_3$;
  \item There is a bounded linear operator $Z\colon \cH_1 \to \cH_2$ satisfying $X=YZ$.
 \end{enumerate}

 \begin{proof}[Proof of (a)$\Rightarrow$(c)]
  By hypothesis, for every $x \in \cH_1$ there is a unique $y \in ( \Ker Y )^\perp$ with $X x = Y y$, and we
  define $Z \colon \cH_1 \to \cH_2$ by $Z x = y$. By construction, this operator $Z$ satisfies $X = YZ$, and it is easy to see that
  it is linear. It remains to show that $Z$ is bounded. Since $Z$ is everywhere defined, by the closed graph theorem it suffices to
  show that $Z$ is closed. To this end, let $(x_n)_n$ be a sequence in $\cH_1$ such that $x_n \to x$ in $\cH_1$ and $Z x_n \to z$ in
  $\cH_2$. Since $X$ and $Y$ are bounded, this yields on the one hand that $X x_n \to X x$ in $\cH_3$ and, on the other hand, that
  $X x_n = YZ x_n \to Yz$ in $\cH_3$, so that $X x = Y z$. Taking into account that $( \Ker Y )^\perp$ is closed and
  $Z x_n \in ( \Ker Y )^\perp$, one has $z \in ( \Ker Y )^\perp$ and, hence, $Z x = z$, which proves the claim.
 \end{proof}%
\end{proposition}

If the system~\eqref{eq:abstrCauchy:app} is null-controllable in time $T>0$, then the implication (a)$\Rightarrow$(c) in
Proposition~\ref{prop:TW} with $X=\ee^{-TH}\colon\cH\to\cH$ and $Y=\cB^T\colon L^2((0,T),\cU)\to\cH$ yields that $F=-Z$ is a bounded
linear feedback operator for this system.

\begin{remark}\label{rem:obs}
 As established in the proof of Lemma~\ref{lem:controllabilitymap}, we have $(\cB^T)^*=B^*\ee^{-(T-\cdot)H}$, so that
 \begin{equation*}
  \norm{ (\cB^T)^*v_0 }_{L^2((0,T),\cU)}^2 = \int_0^T \norm{ B^*\ee^{-(T-s)H}v_0 }_\cU^2 \,\dd s
  = \int_0^T \norm{ B^*\ee^{-tH}v_0 }_\cU^2 \,\dd s
 \end{equation*}
 for all $v_0\in\cH$. The equivalence (a)$\Leftrightarrow$(b) in Proposition~\ref{prop:TW} with the choice $X=\ee^{-TH}$ and
 $Y=\cB^T$ as above therefore yields the equivalence between null-controllability of~\eqref{eq:abstrCauchy:app} and the so-called
 final-state-observability of the corresponding homogeneous system as mentioned in Remark~\ref{rem:alternative}.
\end{remark}

Given any two feedback operators $F$ and $\tilde{F}$ for the system~\eqref{eq:abstrCauchy:app}, one sees analogously
to~\eqref{eq:nullcontrol} that $0 = (\ee^{-TH} + \cB^T F) - (\ee^{-TH} + \cB^T \tilde{F}) = \cB^T(F-\tilde{F})$, that is,
\begin{equation*}
 \Ran (F-\tilde{F}) \subset \Ker \cB^T.
\end{equation*}
Hence, denoting by $P$ the orthogonal projection in $L^2((0,T),\cU)$ onto $\Ker \cB^T$, the operator $\cF^T:=(\Id-P)F$ is again a
feedback operator for the system~\eqref{eq:abstrCauchy:app} and does not depend on the choice of $F$. In particular, one has
$\norm{\cF^T}\le\norm{F}$ for every bounded linear feedback operator $F$. Thus, $\cF^T$ is a bounded linear feedback operator with
minimal operator norm. Moreover, by definition of the orthogonal projection $P$, for every $u_0\in\cH$ one has
\begin{align*}
 \norm{\cF^T u_0}_{L^2((0,T),\cU)}
 &= \norm{Fu_0 - PFu_0}_{L^2((0,T),\cU)}\\
 &= \min\{ \norm{Fu_0 - g}_{L^2((0,T),\cU)} \mid g\in\Ker \cB^T \}\\
 &= \min\{ \norm{f}_{L^2((0,T),\cU)} \mid \ee^{-TH}u_0 + \cB^T f = 0 \},
\end{align*}
where for the last equality we have taken into account that by~\eqref{eq:setofnullcontrols} the set of all null-controls for the
initial datum $u_0$ is given by $Fu_0 + \Ker \cB^T$. Thus, $\cF^T u_0\in (\Ker \cB^T)^\perp$ is the uniquely determined null-control
associated to $u_0$ with minimal norm.

\subsection*{Acknowledgements}
The authors would like to thank M.~Egidi, I.~Naki\'c, M.~T\"aufer, and M.~Tautenhahn
for fruitful discussions and helpful remarks on an earlier version of the manuscript.
I.~V.~thanks the anonymous referee of~\cite{EgidiV-18} for stimulating remarks.


\end{document}